\documentclass[10pt]{amsart}
\usepackage{hyperref}
\usepackage{amsfonts,mathrsfs,bbm,rawfonts,amsmath,amssymb,enumerate,geometry,tikz}
\usepackage{fullpage, setspace}
\usepackage{amscd}
\usepackage{amsthm}
\usepackage{mathrsfs}
\usepackage{amssymb} \usepackage{latexsym}
\usepackage{eufrak}
\usepackage{euscript}
\usepackage{epsfig}
\usepackage{graphics}
\usepackage{array}
\usepackage{enumerate}
\usepackage{dsfont}
\usepackage{color}
\usepackage{wasysym}
\usepackage{hyperref}
\usepackage{pdfsync}

\numberwithin{equation}{section}
\allowdisplaybreaks
\newcommand{\RNum}[1]{\uppercase\expandafter{\romannumeral #1\relax}}

\newcommand{\R}{{\mathbb R}}

\newcommand{\N}{{\mathbb N}}
\newcommand{\Z}{{\mathbb Z}}

\newcommand{\HH}{{\mathcal{H}}}

\newtheorem{theorem}{Theorem}[section]
\newtheorem{lem}[subsection]{Lemma}
\newtheorem{thm}[theorem]{Theorem}
\newtheorem{pro}{Proposition}[section]
\newtheorem{cor}[theorem]{Corollary}
\newtheorem{defi}[theorem]{Definition}
\newtheorem{rem}[theorem]{Remark}
\newtheorem{exa}[theorem]{Example}

\newcommand{\Hmm}[1]{\leavevmode{\marginpar{\tiny%
$\hbox to 0mm{\hspace*{-0.5mm}$\leftarrow$\hss}%
\vcenter{\vrule depth 0.1mm height 0.1mm width \the\marginparwidth}%
\hbox to
0mm{\hss$\rightarrow$\hspace*{-0.5mm}}$\\\relax\raggedright #1}}}

\title
{Steklov eigenvalue problem on subgraphs of integer lattices}
\author{Wen Han}
\address{Wen Han: School of Mathematical Sciences, Fudan University, Shanghai 200433, China.}
\email{wenhan122@foxmail.com}

\author{Bobo Hua}
\address{Bobo Hua: School of Mathematical Sciences, LMNS, Fudan University, Shanghai 200433,
China; Shanghai Center for Mathematical Sciences, Fudan University, Shanghai 200433, China}
\email{bobohua@fudan.edu.cn}

\begin{document}


\begin{abstract}
\setlength{\parindent}{0pt} \setlength{\parskip}{1.5ex plus 0.5ex
minus 0.2ex} 

We study the eigenvalues of the Dirichlet-to-Neumann operator on a finite subgraph of the integer lattice $\Z^n.$ We estimate the first $n+1$ eigenvalues using the number of vertices of the subgraph. As a corollary, we prove that the first non-trivial eigenvalue of the Dirichlet-to-Neumann operator tends to zero as the number of vertices of the subgraph tends to infinity.
\end{abstract}

\maketitle

\section{Introduction}

Given a bounded smooth domain $\Omega$ in $\mathbb{R}^n,$ we consider the Steklov (eigenvalue) problem on $\Omega.$ For a smooth function {$\varphi $} on $\partial \Omega,$ we denote by $u_{\varphi}$ the harmonic extension of $\varphi $ into $\Omega$ which satisfies 
\[\left\{\begin{array}{ll} \Delta u_{\varphi}(x)=0,& x\in \Omega,\\
u_{\varphi}|_{\partial \Omega}=\varphi.& \end{array}\right.\] The Steklov problem on $\Omega$ is the following eigenvalue problem,
$$\Lambda \varphi:=\frac{\partial u_{\varphi}}{\partial n}=\lambda \varphi \quad \mathrm{on}\ \partial \Omega,$$ where $n$ denotes the unit outward normal of $\partial\Omega.$
The operator $\Lambda$ is called the Dirichlet-to-Neumann (DtN in short) operator, which is also known as the voltage-current map in physics, see e.g. Calder\'{o}n's problem \cite{Calderon2006}. As a pseudo-differential operator, the eigenvalues of DtN operator $\Lambda,$ called Steklov eigenvalues,  are discrete and can be ordered as 
$$0=\sigma_1(\Omega)\leq \sigma_2(\Omega)\leq \cdots\leq \sigma_k(\Omega)\leq \cdots \uparrow \infty.$$
The Steklov problem has been intensively studied for domains in Euclidean spaces and Riemannian manifolds in the literature, see e.g. \cite{Weinstock1954,ESCOBAR1997,Escobar1999,BrockF2001,HenrotPhilippinSafoui2008,WANG2009,Sylvester2010,COLBOIS2011,HASSANNEZHAD2011,GirouardPolterovich2012,fraser2014,KarpukhinKokarevPolterovich2014,GirouardLaugesenSiudeja2016,Fraser2016,Matevossian2017,FallWeth2017,YangYu2017}.

For $n=2,$ Weinstock \cite{Weinstock1954} proved that for any planar simply-connected domain $\Omega$ with analytic boundary, 
\begin{equation}\label{Wei}
\sigma_2(\Omega)\leq \frac{2\pi}{\mathrm{L}(\partial\Omega)},
\end{equation} 
where $\mathrm{L}(\partial\Omega)$ denotes the length of the boundary $\partial\Omega.$ The statement was generalized to domains with Lipschitz boundary by Girouard and Polterovich\cite{girouard2010}. The higher dimensional generalization was proved by Bucur, Ferone, Nitsch and Trombetti \cite{bucur2017}: for any bounded convex $\Omega\subset \R^n,$ $n\geq 3,$
\[
\sigma_2(\Omega) (\mathrm{area}(\partial\Omega))^{\frac{1}{n-1}}  \leq \sigma_2(B) (\mathrm{area}(\partial B))^{\frac{1}{n-1}},
\] where $B$ is the unit ball and $\mathrm{area}(\cdot)$ denotes the area of $(\cdot),$
and the equality holds only for balls, see also Fraser and Schoen \cite{fraser2019} for related results.

The estimate (\ref{Wei}) was improved by Hersch and Payne \cite{HerschPayne1968} to the following
\begin{equation}\label{HP}
\frac{1}{\sigma_2(\Omega)}+\frac{1}{\sigma_3(\Omega)}\geq \frac{\mathrm{L}(\partial\Omega)}{\pi}.
\end{equation}
Concerning with the Steklov eigenvalues and the volume of the domain, Brock proved the following result.

\begin{theorem}[Theorem~3 in \cite{BrockF2001}]\label{thm:b1} Let $\Omega$ be a smooth domain in $\R^n.$ Then
$$\sum_{i=2}^{n+1}\frac{1}{\sigma_i(\Omega)}\geq C(n)(\mathrm{vol}(\Omega))^{\frac1n},$$ where $\mathrm{vol}(\Omega)$ denotes the volume of $\Omega,$ and $C(n)=n\omega_n^{-\frac1n}$ (here $\omega_n$ is the volume of the unit ball in $\R^n$).  
\end{theorem}

In this paper, we study the Steklov problem on graphs and prove a discrete analog for Brock's result. The DtN operator on a subgraph of a graph was introduced by \cite{BoboYanZuoqin2017,Miclo2017} independently, see e.g. \cite{BoboYanZuoqin2018,Perrin2018}
 for more results. A graph $G=(V,E)$ consists of the set of vertices $V$ and the set of edges $E.$ In this paper, we only consider simple, undirected graphs. Two vertices $x,y$ are called neighbours, denoted by $x\sim y$, if there is an edge $e$ connecting $x$ and $y,$ i.e. $e=\{x,y\}\in E.$ Integer lattice graphs are of particular interest which serve as the discrete counterparts of $\R^n.$ We denote by $\Z^n$ the set of integer $n$-tuples $\R^n.$ The $n$-dimensional integer lattice graph, still denoted by $\Z^n,$ is the graph consisting of the set of vertices $V=\Z^n$ and the set of edges $$E=\left\{\{x,y\}: x,y\in \Z^n, \sum_{i=1}^n|x_i-y_i|=1\right\}.$$
Let $\Omega$ be a finite subset of $\Z^n.$ We denote by  $$\delta \Omega:=\{x \in \Z^n\setminus \Omega :  \exists\ y\in \Omega,s.t. \  x\sim y \}$$  the vertex boundary of $\Omega.$ Analogous to the continuous setting, one can define the DtN operator on $\Omega,$ $$\Lambda:\R^{\delta\Omega}\to\R^{\delta\Omega},$$ where $\R^{\delta\Omega}$ denotes the set of functions on $\delta\Omega,$ see Section~\ref{sec:2} for details. In this paper, we denote by $|\cdot|$ the cardinality of a set.
Note that $\Lambda$ can be written as a symmetric matrix whose eigenvalues are ordered as 
\begin{equation}
0=\lambda_1(\Omega)\leq \lambda_2(\Omega)\leq \cdots\leq \lambda_{N}(\Omega), \ \text{where}~N=|\delta\Omega|. \label{order}
\end{equation} 

The following is our main result, which is a discrete analog of \cite[Theorem~3]{BrockF2001}. 




\begin{theorem}\label{mainthm:1} Let $\Omega$ be a finite subset of $\Z^n.$ Then 
\begin{equation}
\sum^{n+1}_{i=2}\frac{1}{\lambda _i(\Omega)}
\geq C_1|\Omega|^{\frac1n}-\frac{C_2}{|\Omega|}, 
\end{equation} where $\lambda_i(\Omega)$ are Steklov eigenvalues on $\Omega,$ $C_1=(64 n^3 \omega_n^{\frac1n})^{-1},$ and $C_2=\frac{1}{32n}.$

\end{theorem}

\begin{rem} By this theorem, for $|\Omega|\geq (\frac{2C_2}{C_1})^{\frac{n}{n+1}},
$
$$\sum^{n+1}_{i=2}\frac{1}{\lambda _i(\Omega)}
\geq \frac12 C_1|\Omega|^{\frac1n},$$
where $(\frac{2C_2}{C_1})^{\frac{n}{n+1}}={(4n^2\omega_n^{\frac1n})^{\frac{n}{n+1}}}$ is of order $O(n^{\frac{3}{2}}), n\to\infty.$
\end{rem}

The main ingredient of the proof of Brock's result is a weighted isoperimetric inequality, see Lemma~\ref{lem:rearr} below, which depends on the rotational symmetry of the Euclidean spaces. As is well-known in the discrete theory, the symmetrization approaches do not work on $\Z^n$ due to the lack of rotational symmetry. In this paper, we follow Brock's idea to bound the Steklov eigenvalues by the geometric quantities of the subset in the lattice $\Z^n,$ and then estimate these discrete quantities by their counterparts in the Euclidean space $\R^n$, for which we can apply the symmetrization approach. 
Through this process we obtain the quantitative estimate of discrete Steklov eigenvalues, but lose the sharpness of the constants $C_1$ and $C_2$ in our result.

As a corollary, we obtain the estimate for the first non-trivial eigenvalue of the DtN operator.
\begin{cor}Let $\Omega$ be a finite subset of $\Z^n.$ Then
\begin{equation}\label{cor1}
\lambda _2(\Omega)\leq \frac{n}{C_1|\Omega|^{\frac1n}-\frac{C_2}{|\Omega|}},
\end{equation} where $C_1$ and $C_2$ are the constants in Theorem~\ref{mainthm:1}.
\end{cor}
\begin{rem} Recently, the above corollary has been extended to the DtN operators on subgraphs of Cayley graphs for discrete groups of polynomial growth by Perrin \cite{Perrin2020}.
\end{rem}
This yields an interesting consequence that for any sequence of finite subsets in $\Z^n,$ $\{\Omega_i\}_{i=1}^\infty,$ satisfying $|\Omega_i|\to \infty,$ \begin{equation}\label{eq:to0}\lambda _2(\Omega_i)\to 0,\quad i\to \infty.\end{equation}  
This is a rather special property for integer lattices, which does not hold for general infinite graphs. At the end of the paper, we give an example of an infinite graph of bounded degree, Example~\ref{tree}, in which there exist subsets $\Omega_i$ with $|\Omega_i|\to \infty$ such that
$$\lambda _2(\Omega_i)\geq c>0,\quad \forall i.$$ Note that this graph locally behaves similarly as a homogeneous tree of degree $3,$ which is a discrete hyperbolic model.

The paper is organized as follows: In the next section, we recall some facts on DtN operators on graphs. In Section~\ref{sec:3}, we prove Theorem~\ref{mainthm:1}.

\section{Preliminaries}\label{sec:2}
Let $G=(V,E)$ be a simple, undirected graph.  
For two subsets $A, B$ of $V$, we define \[
E(A, B):=\{\{x,y\}\in E:~x \in A,~y \in B,\ \mathrm{or\ vice\ versa}\},
\] which is a set of edges connecting a vertex in $A$ to another vertex in $B.$
For a subset $\Omega$ of $V$, we write $\Omega^{c}:=V\setminus \Omega.$ The edge boundary of $\Omega$ is defined as $\partial \Omega:=E(\Omega,\Omega^{c}),$ and we set $\bar{\Omega}:=\Omega \cup \delta \Omega$. 

For a finite subset $S$ of $V,$ we denote by $\R^S$ the set of functions on $S,$ by $\ell^2(S)$ the Hilbert space of functions on $S$ equipped with the inner product $$\langle \varphi, \psi \rangle_{S}=\sum_{x\in S}\varphi(x)\psi(x), \quad \varphi, \psi\in\mathbb{R}^{S}.$$ In the following, we will take $S=\Omega,$ $\delta\Omega,$ or $\bar{\Omega},$ respectively. 
For $u\in\R^{\bar{\Omega}},$ 
the Dirichlet energy of $u$ is defined as
\[
D_{\Omega}(u):=\sum_{\{x,y\}\in E(\Omega, \bar{\Omega})}(u(x)-u(y))^2.
\] The polarization of the Dirichlet energy is given by
\begin{equation}\label{energy}
D_{\Omega}(u,v):=\sum_{\{x,y\}\in E(\Omega, \bar{\Omega})}(u(x)-u(y))(v(x)-v(y)),\ \forall u,v\in\R^{\bar{\Omega}}.
\end{equation}

The Laplace operator on $\Omega$ is defined as
\begin{align*}
\Delta :\R^{\bar{\Omega}} & \rightarrow \R^{\Omega} \\ 
u&   \mapsto   \Delta u(x):=\sum_{y\in V: y\sim x}(u(y)-u(x)).
\end{align*}
The exterior normal derivative on vertex boundary set $\delta \Omega$ is defined as
\begin{align*}
\frac{\partial }{\partial n} :\R^{\bar{\Omega}} & \rightarrow \R^{\delta\Omega} \\ 
u&   \mapsto   \frac{\partial u}{\partial n} (x):=\sum_{y\in \Omega: y\sim x}(u(x)-u(y)),\ x\in \delta\Omega.
\end{align*}
The Dirichlet-to-Neumann (DtN in short) operator is defined as
\begin{align*}
\Lambda :\R^{\delta\Omega} & \rightarrow \R^{\delta\Omega} \\ 
\varphi&   \mapsto   \Lambda \varphi := \frac{\partial u_{\varphi}}{\partial n} ,
\end{align*}
where $u_{\varphi}$ is the harmonic function on $\Omega$ whose Dirichlet boundary condition on $\delta\Omega$ is given by $\varphi.$ 
We recall the well-known Green formula. 

\begin{thm}[Lemma 2.1 in \cite{BoboYanZuoqin2017}]\label{thm:foruse}For $\varphi\in \mathbb{R}^{\delta{\Omega}},$
\begin{equation} \label{green}
D_{\Omega}(u_{\varphi})=-\langle\Delta u_{\varphi}, u_{\varphi}\rangle_{\Omega}+\langle\frac{\partial u_{\varphi}}{\partial n},u_{\varphi}\rangle_{\delta\Omega},
\end{equation} where $u_{\varphi}$ is the harmonic extension of $\varphi$ in $\Omega.$
\end{thm}
\begin{pro}
The DtN operator $\Lambda$ can be represented as a symmetric matrix, and hence is diagonalizable. The multiplicity of zero-eigenvalues of $\Lambda$ is equal to the number of connected components of the graph $(\bar{\Omega}, E(\Omega, \bar{\Omega}))$.
\end{pro}
\begin{proof}[Proof]

Assume that $\phi, \varphi \in\mathbb{R}^{\delta\Omega}$ are eigenfunctions, which can be regarded as vectors in $\R^{N}$ with $N=|\delta\Omega|,$ of the DtN operator $\Lambda,$ which is a $N\times N$ matrix. Using the polarization of the Dirichlet energy (\ref{energy}) and Green formula (\ref{green}), we have
$$\langle\Lambda u_{\phi}, u_{\varphi}\rangle_{\delta\Omega}= D_{\Omega}(u_{\phi}, u_{\varphi}) = D_{\Omega}(u_{\varphi}, u_{\phi})= \langle u_{\phi}, \Lambda u_{\varphi}\rangle_{\delta\Omega}.$$
Hence, $\Lambda$ is a symmetric matrix, and hence is diagonalizable.

Assume that $\varphi\in\mathbb{R}^{\delta\Omega}$ is an eigenfunction corresponding to the eigenvalue 0 of the DtN operator $\Lambda,$ then by the Green formula (\ref{green}),
\[
D_{\Omega}(u_{\varphi})=-\langle\Delta u_{\varphi}, u_{\varphi}\rangle_{\Omega}+\langle\frac{\partial u_{\varphi}}{\partial n},u_{\varphi}\rangle_{\delta\Omega}=0.
\]
Thus, $u_{\varphi}$ is constant on every connected component of  $(\bar{\Omega}, E(\Omega, \bar{\Omega})).$ This yields the result.
\end{proof}
Let $\lambda$ be an eigenvalue of the DtN operator $\Lambda$ satisfying $$\frac{\partial u_{\varphi}}{\partial n} =\lambda u_{\varphi}\ \mathrm{on \ \delta \Omega}.$$ Then, by the Green formula (\ref{green}) we get
\[
 D_{\Omega}(u_{\varphi})=\lambda \langle u_{\varphi},u_{\varphi}\rangle_{\delta\Omega}=\lambda \sum_{z \in \delta\Omega}{\varphi}^2(z).
\]
The following variational principle is useful, see e.g. \cite[p.99]{Bandle1980}.

\begin{thm}\label{thm:variation}
 Let $\Omega$ be a finite subset of $\Z^n.$ Then
%

\begin{equation}
 \sum^{p}_{i=2}\frac{1}{\lambda _i(\Omega)}=
 \max\left\{ \sum^{p}_{i=2}\sum_{z\in \delta\Omega}v_i^2(z): v_i\in \R^{\bar\Omega}, D_{\Omega}(v_i,v_j)=\delta_{ij}, \sum_{z\in \delta\Omega}v_i(z)=0, \ 2\leq i, j\leq p \right\},
\end{equation} where $p \geq 2$ and $\delta_{ij}=\left\{\begin{array}{ll} 1,& i=j,\\
0,& i\neq j,\end{array}\right.$ where we adopt the traditional convention that $\lambda_{l}=+\infty \text { for } l>|\delta \Omega|$.
\end{thm}

\section{The proof of Theorem~\ref{mainthm:1}}\label{sec:3} In this section, we prove the main theorem.
Let $\Z^n$ be the $n$-dimensional integer lattice graph, and $$e_i:=(0,\cdots,\underset{i-1}{0},\underset{i}{1}, \underset{i+1}{0},\cdots 0),\quad 1\leq i\leq n,$$  be the standard basis of $\R^n.$ We denote by $$Q=\{x=(x_1,x_2,\cdots,x_n)\in \R^n: |x_i|\leq \frac12,\ \forall 1\leq i\leq n\}$$ the unit cube centered at the origin in $\R^n,$ and by
$$Q_i:=Q\cap\{x\in \R^n: x_i=0\},\quad 1\leq i\leq n,$$ the $(n-1)$-dimensional unit cube in $x_i$-hyperplane. For any edge $\tau=\{x,y\}$ in $\Z^n,$ there is a unique $e_i$ such that $y-x=e_i$ or $-e_i,$  and we define
$$\tau^{\perp}:=\frac12(x+y)+Q_i.$$ That is, $\tau^{\perp}$ is an $(n-1)$-dimensional unit cube centered at the point $\frac12(x+y)$ parallel to $x_i$-hyperplane. Note that such a cube $\tau^\perp$ is one-to-one corresponding to $\tau.$

For any finite subset $\Omega\subset \Z^n,$ we denote $$\partial \Omega^\perp :=\{\tau^\perp: \tau\in\partial\Omega\}.$$ Note that for any $\tau\in \partial \Omega,$ there is a unique end-vertex of the edge $\tau$ belonging to $\delta\Omega.$ We define a mapping 
\begin{align*}
P:\partial \Omega^\perp  & \rightarrow \delta\Omega \\ 
\tau^\perp&   \mapsto  P(\tau^\perp),
\end{align*} where $P(\tau^\perp)$ is the end-vertex of $\tau$ in $\delta\Omega.$
\begin{lem} For any $\tau_1$,$\tau_2\in \partial\Omega$ satisfying $P(\tau_1^\perp)\neq P(\tau_2^\perp),$ we have
\[
\int_{\tau_1^\perp}\int_{\tau_2^\perp}|s-t|^2d\HH^{n-1}(s)d\HH^{n-1}(t)\leq C_3\cdot |P(\tau_1^\perp)- P(\tau_2^\perp)|^2,
\] where $C_3=4n,$ and $\HH^{n-1}$ denotes the $(n-1)$-dimensional Hausdorff measure in $\R^n.$\label{L1}
\end{lem}
\begin{proof}[Proof]
Take $\tau_j=\{x_j, y_j\},$ $j=1,2,$ such that there exists $e_{i_j}$ satisfying $y_j-x_j=e_{i_j}.$ By the definition $\tau^{\perp}_1=\frac12(x_1+y_1)+Q_{i_1}.$ By the symmetry, without loss of generality we may assume that $P(\tau_1^\perp)=x_1.$ For any $s \in \tau_1^\perp,$ we write $s=\frac12(x_1+y_1)+q_{i_1}$ where $q_{i_1}\in Q_{i_1}.$ Then, 
\[
|s-P(\tau_1^\perp)|=|\frac12(x_1+y_1)+ q_{i_1}-x_1|=|\frac{e_{i_1}}{2}+q_{i_1}|\leq \frac{\sqrt{n}}{2}.
\]
Similarly, we get  for any $ t \in \tau_2^\perp$
\[
|t-P(\tau_2^\perp)|\leq \frac{\sqrt{n}}{2}.
\]
Since $|P(\tau_1^\perp)-P(\tau_2^\perp)|\geq 1$, by the triangle inequality, for any $s \in \tau_1^\perp, t \in \tau_2^\perp$ we have
\begin{eqnarray*}
|s-t|&\leq& |s-P(\tau_1^\perp)|+|P(\tau_1^\perp)-P(\tau_2^\perp)|+|P(\tau_2^\perp)-t|
\\&\leq&(\sqrt{n}+1)|P(\tau_1^\perp)-P(\tau_2^\perp)|.
\end{eqnarray*}
Hence,
\begin{eqnarray*}
\int_{\tau_1^\perp}\int_{\tau_2^\perp}|s-t|^2d\HH^{n-1}(s)d\HH^{n-1}(t)&\leq& (\sqrt{n}+1)^2 |P(\tau_1^\perp)- P(\tau_2^\perp)|^2\\&\leq& 4n |P(\tau_1^\perp)- P(\tau_2^\perp)|^2.
\end{eqnarray*}
This proves the lemma.
\end{proof}

For any $r>0,$ we denote by $$Q_r(x):=\{y\in \R^n:\max_{1\leq i\leq n}|y_i-x_i|\leq \frac{r}{2}\}$$ the cube of side length $r$ centered at  $x.$
For any subset $\Omega\subset \Z^n,$ we construct a domain $\hat{\Omega}$ in $\R^n$ as follows
\[
\hat{\Omega}:=\bigcup_{x \in\Omega}Q_1(x).
\]
Note that the topological boundary of $\hat{\Omega},$ denoted by $\partial(\hat{\Omega}),$ is one-to-one corresponding to $\partial \Omega^\perp,$ i.e. \begin{equation}\label{eq:pho}\partial(\hat{\Omega})=\bigcup_{\tau\in\partial \Omega}\tau^{\perp}.\end{equation} For our purposes,  we will use the geometry of $\hat{\Omega}$, which is a domain in $\R^n$, to estimate that of $\Omega,$ a subset in $\Z^n.$

We give the definition of ``bad" boundary vertices on $\delta\Omega.$
\begin{defi}[``Bad" boundary vertex]
We call $x\in \delta\Omega$ a ``bad" boundary vertex if it is isolated by the set $\Omega,$ i.e. for any $y\sim x$ in $\Z^n,$ $y\in \Omega.$ The set of ``bad" boundary vertices is denoted by $\delta'\Omega.$
\end{defi}

In the following, we give an example to present the role of ``bad" boundary vertices in our estimates, which motivates us to distinguish boundary vertices. In the proof of Theorem~\ref{mainthm:1}, following Brock \cite{BrockF2001} we reduce the estimate of Steklov eigenvalues to the lower bound estimate of $\sum_{z\in \delta\Omega}\sum_{\omega\in \delta\Omega}|z -\omega|^2,$ see \eqref{main} below. The strategy is to find its counterpart in the continuous setting for deriving such a lower bound, which can be estimated by isoperimetric inequalities in $\R^n.$ A natural choice is the quantity 
$$\int_{\partial(\hat{\Omega})}\int_{\partial(\hat{\Omega})}|s-t|^2d\HH^{n-1}(s)d\HH^{n-1}(t)=2\HH^{n-1}(\partial(\hat{\Omega}))\int_{\partial(\hat{\Omega})}|s-c|^2d\HH^{n-1}(s),$$
where $c=\frac{1}{\HH^{n-1}(\partial(\hat{\Omega}))}\int_{\partial(\hat{\Omega})}sd\HH^{n-1}(s)$ is the barycenter of $\partial(\hat{\Omega})$ and the right hand side is the derivation of the position vector of $\partial(\hat{\Omega})$ with respect to the measure $\HH^{n-1}.$  For the desired estimate
\begin{equation}\label{eq:ep1}\sum_{z\in \delta\Omega}\sum_{\omega\in \delta\Omega}|z -\omega|^2\geq C\int_{\partial(\hat{\Omega})}\int_{\partial(\hat{\Omega})}|s-t|^2d\HH^{n-1}(s)d\HH^{n-1}(t)+\cdots,\end{equation} we give the following example to indicate the difficulty.
\begin{exa}\label{ex:pp1} For $\Z^2,$ let $$\Omega:=\{(x_1,x_2)\in \Z^2 | \max\{|x_1|,|x_2|\}\leq 2, (x_1,x_2)\neq (0,0)\},$$ see Figure~\ref{fig378}.
Set $A_1:=\delta'\Omega=\{(0,0)\}, A_2:=\delta\Omega\setminus \delta'\Omega,$
$B_1:=\partial Q,$ the boundary of the unit cube, and $B_2:=\partial (Q_5((0,0))).$ Then $\partial(\hat{\Omega})=B_1\cup B_2.$ We have the following
\begin{eqnarray*}
\sum_{z\in \delta\Omega}\sum_{\omega\in \delta\Omega}|z -\omega|^2=
\left(\sum_{z\in A_1}\sum_{\omega\in A_1}+2\sum_{z\in A_1}\sum_{\omega\in A_2}+\sum_{z\in A_2}\sum_{\omega\in A_2}\right)|z -\omega|^2,
\end{eqnarray*}
\begin{eqnarray*}
\int_{\partial(\hat{\Omega})}\int_{\partial(\hat{\Omega})}|s-t|^2d\HH^{1}(s)d\HH^{1}(t)=
\left(\int_{B_1}\int_{B_1}+2\int_{B_1}\int_{B_2}+\int_{B_2}\int_{B_2}\right)|s-t|^2d\HH^{1}(s)d\HH^{1}(t).
\end{eqnarray*} For deriving the estimate \eqref{eq:ep1}, it is impossible to bound the term
$$\int_{B_1}\int_{B_1}|s-t|^2d\HH^{1}(s)d\HH^{1}(t)>0$$ from above by the term $\sum_{z\in A_1}\sum_{\omega\in A_1}|z -\omega|^2=0.$ The other terms are mutually comparable. This suggests that ``bad" boundary vertices, $A_1$ here, are the obstructions for the estimate \eqref{eq:ep1}.\end{exa}

 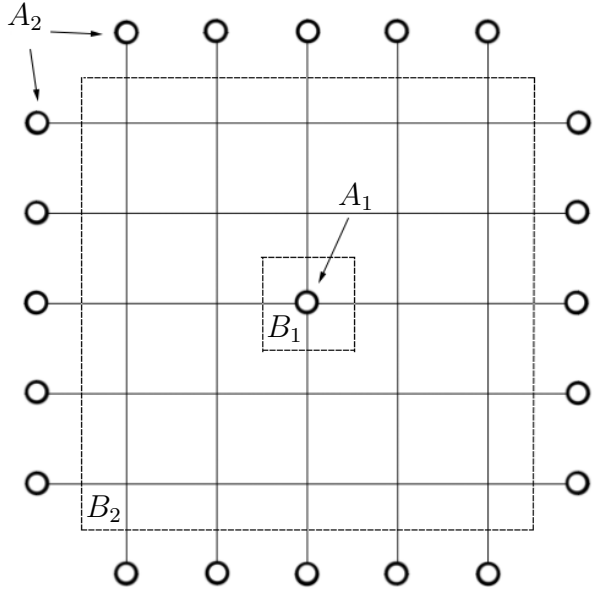
\begin{figure}[htbp]
 \begin{center}
   \begin{tikzpicture}
    \node at (0,0){\includegraphics[width=0.5\textwidth]{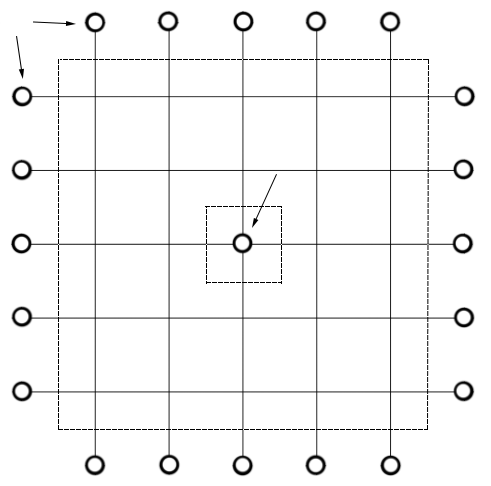}};
    \node at (-2.75,   -2.75){\Large $B_2$};
    \node at (-0.35,   -.35){\Large $B_1$};
    \node at (-3.8,  3.8){\Large $A_2$};
    \node at (0.6,  1.4){\Large $A_1$};
   \end{tikzpicture}
  \caption{$A_i, B_i~(i=1, 2)$ on $\Omega.$}
\label{fig378}
 \end{center}
\end{figure}

\begin{pro} For any finite subset $\Omega$ in $\Z^n,$ 
$$|\delta'\Omega|\leq |\Omega|.$$ 
\label{voln}
\end{pro}
\begin{proof}[Proof of Proposition~\ref{voln}]We define a mapping 
\begin{align*}
F:\delta'\Omega   & \rightarrow \Omega\\ 
x&   \mapsto  F(x):=x+e_1.
\end{align*} 
It is easy to see that $F$ is injective, which yields $|\delta'\Omega|=|F(\delta'\Omega)|\leq |\Omega|.$ This proves the proposition. 
\end{proof}

The next example shows that the above estimate is sharp.
\begin{exa}For $R\in \N,$ set $$\Omega_R:= \{(x, y)\in \Z^2: R=|x|\geq |y|~\text{or}~R=|y|\geq |x|\}\cup \{(x, y)\in \Z^2: |x|\leq R-1,~|y|\leq R-1~\mathrm{s.t.}\ x+y\ \mathrm{is\ odd}\}.$$ 
Then $$\lim_{R\to\infty} \frac{|\delta'\Omega_R|}{|\Omega_R|}=1.$$ 
\end{exa}

For any subset $\Omega\subset \Z^n,$ any vertex in $V\setminus \bar{\Omega}$ is called an exterior vertex of $\Omega.$ For simplicity, we write $\Omega^e:=V\setminus \bar{\Omega}$ for the set of exterior vertices of $\Omega.$ This yields the decomposition 
\begin{equation}\label{decomposition}
\Z^n= \Omega \sqcup \delta\Omega \sqcup \Omega^e,
\end{equation} 
where $\sqcup$ denotes the disjoint union.

\begin{lem}
Let $\Omega$ be a finite subset of $\Z^n.$ For any $x \in \delta\Omega\setminus \delta'\Omega,$ $$|Q_2(x)\cap \delta\Omega|\geq 2.$$
\label{L2ofn}
\end{lem}
\begin{proof}[Proof]
For any $x \in \delta\Omega,$ set $N(x):=\{y\in \Z^n: y\sim x\}.$ Without loss of generality, we may assume that $N(x)\cap \delta \Omega= \emptyset,$ otherwise the lemma follows trivially. 
Thus, with the decomposition (\ref{decomposition}), we have $N(x) \subset \Omega \cup \Omega^{e}.$ Moreover, by the assumption $x\not\in\delta'\Omega ,$ we get $N(x) \not\subset \Omega$ so that
\begin{equation}\label{eq:hhh1}
N(x)\cap \Omega^e\neq \emptyset,\ \mathrm{and}\ N(x)\cap \Omega\neq \emptyset.
\end{equation}
For any $1\leq i<j\leq n,$ we set $N_{ij}=\{x\pm e_i, x\pm e_j\},$ which is the intersection of $N(x)$ with the $2$-dimensional affine plane passing through $x$ spanned by the direction $\{e_i,e_j\}.$ We claim that there exist $i_0$ and $j_0,$ $1\leq i_0<j_0\leq n,$ such that $$N_{i_0j_0}\cap \Omega^e\neq \emptyset,\ \mathrm{and}\ N_{i_0j_0}\cap \Omega\neq \emptyset.$$ Suppose that it is not true, then for any pair $\{i,j\},$ $1\leq i<j\leq n,$ $N_{ij}\subset \Omega$ or $N_{ij}\subset \Omega^e.$ Without loss of generality, we assume that $N_{12}\subset \Omega.$ Then $x\pm e_1\in \Omega,$ which yields that
$N_{1j}\subset \Omega$ for any $j\geq 2$ by the above condition. This implies that $x\pm e_j\in \Omega$ for any $1\leq j\leq n$ and hence $N(x)\subset \Omega,$ which contradicts \eqref{eq:hhh1}. This proves the claim. By the claim, for $N_{i_0j_0},$ one easily verifies that there exists $y\in N_{i_0j_0}\cap \Omega^e\subset N(x)\cap \Omega^e$ and $z\in N_{i_0j_0}\cap \Omega\subset N(x)\cap \Omega$ such that
$$|y-z|=\sqrt2.$$ 
Set $w:=y+z-x.$ Then $$w\in Q_2(x),\ y\sim w,\ z\sim w.$$ Note that any path in $\Z^n$ connecting a vertex in $\Omega^e$ and a vertex in $\Omega$ contains a vertex in $\delta\Omega.$ Applying this to the path $y\sim w\sim z,$ we get $w\in \delta\Omega.$ This proves the lemma.
\end{proof}

For our purposes, we classify pairs of boundary edges, $\partial \Omega\times \partial \Omega,$ into ``good" pairs and ``bad" pairs. The motivation is similar to that for ``bad" boundary vertices, see Example~\ref{ex:pp1}.
\begin{defi} A pair of boundary edges $(\tau_1,\tau_2)\in \partial \Omega\times \partial \Omega$ is called ``good" if $P(\tau^{\perp}_1)\neq  P(\tau^{\perp}_2)$ or $P(\tau^{\perp}_1)=  P(\tau^{\perp}_2)\notin\delta'\Omega.$ Otherwise, it is called ``bad". We denote by $(\partial\Omega)^2_g$ the set of ``good" pairs of boundary vertices, and by $(\partial\Omega)^2_b$ the set of ``bad" pairs of boundary vertices.
\end{defi}
By the definition, \begin{equation}\label{eq:gb}\partial \Omega\times \partial \Omega=(\partial\Omega)^2_g\sqcup (\partial\Omega)^2_b\end{equation} and $$(\partial\Omega)^2_b=\{(\tau_1, \tau_2)\in \partial\Omega \times  \partial\Omega~|P(\tau^{\perp}_1)=  P(\tau^{\perp}_2)\in\delta'\Omega\}.$$

\begin{defi}[The multiplicity of a mapping]
Let $T: A\to B$ be a mapping. The multiplicity of the mapping $T$ is defined as
\[
m_T:=\sup_{b\in B}|\{T^{-1}(b)\}|.
\] 
\end{defi} 

We estimate the number of ``bad" pairs of boundary vertices by the number of ``bad" vertices.
\begin{pro}\label{prop:bpv1}Let $\Omega$ be a finite subset of $\Z^n.$ Then $$|(\partial\Omega)^2_b|\leq4n^2|\delta'\Omega|. $$
\end{pro}
\begin{proof}We define a mapping
\begin{align*}
H:(\partial\Omega)^2_b  & \rightarrow \delta'\Omega\\ 
(\tau_1,\tau_2)&   \mapsto H(\tau_1,\tau_2)=P(\tau_1^\perp).
\end{align*} We estimate the multiplicity of the mapping $H.$ Fix any $x\in \delta'\Omega.$ For any $(\tau_1,\tau_2)\in (\partial\Omega)^2_b $ such that $H(\tau_1,\tau_2)=x,$ 
$$P(\tau_1^\perp)=P(\tau_2^\perp)=x.$$ So that $x$ is a common end-vertex of $\tau_1$ and $\tau_2.$ The number of possible pairs $(\tau_1,\tau_2)$ satisfying this property is at most $4n^2.$
Hence $m_T\leq 4n^2.$ This proves the proposition.
\end{proof}

We define a mapping on``good" pairs of boundary vertices as
\begin{align}\label{def:f}
f:(\partial\Omega)^2_g& \rightarrow \delta\Omega \times \delta\Omega\\ 
(\tau_1 ,\tau_2)&   \mapsto   f(\tau_1 ,\tau_2)=(f_1(\tau_1 ,\tau_2), f_2(\tau_1 ,\tau_2)):=\begin{cases}
 (P(\tau_1^\perp),P(\tau_2^\perp)),&\text{if  } P(\tau_1^\perp)\neq P(\tau_2^\perp) \\ 
 (P(\tau_1^\perp), z), &\text{if  }  P(\tau_1^\perp)=P(\tau_2^\perp)\not\in \delta'\Omega,
\end{cases}\nonumber
\end{align}
where $z\in Q_2(P(\tau_1^\perp))\cap \delta\Omega$ such that $z\neq P(\tau_1^\perp).$
\begin{rem} In the definition of the mapping $f,$ the existence of $z\in Q_2(P(\tau_1^\perp))\cap \delta\Omega$ in the second case, which is not necessarily unique, is guaranteed by Lemma \ref{L2ofn}. Any choice of such a vertex $z$ will be sufficient for our purposes.
\end{rem}

\begin{lem}\label{fC1} Let $\Omega$ be a finite subset of $\Z^n.$ Then for any $(\tau_1,\tau_2)\in (\partial\Omega)^2_g,$
\[
\int_{\tau_1^\perp}\int_{\tau_2^\perp}|s-t|^2d\HH^{n-1}(s)d\HH^{n-1}(t)\leq C_3\cdot |f_1(\tau_1 ,\tau_2)- f_2(\tau_1 ,\tau_2)|^2,
\] where $C_3$ is the constant in Lemma~\ref{L1}.
\end{lem}

\begin{proof}[Proof of Lemma~\ref{fC1}] By Lemma \ref{L1}, it suffices to consider the case that $P(\tau_1^\perp)= P(\tau_2^\perp)\not\in \delta'\Omega.$ Set $X:=P(\tau_1^\perp)= P(\tau_2^\perp).$ Then $\tau_1\cap \tau_2=X.$ 
By the triangle inequality, for any $ s\in \tau_1^\perp,$ $ t\in \tau_2^\perp,$ we have
\begin{equation}
|s-t| \leq |s-X|+|X-t|\leq\sqrt{n}.\label{3.3}
\end{equation}
Then $\int_{\tau_1^\perp}\int_{\tau_2^\perp}|s-t|^2d\HH^{n-1}(s)d\HH^{n-1}(t)\leq n.$
Obviously, since $|f_1(\tau_1 ,\tau_2)- f_2(\tau_1 ,\tau_2)|=|P(\tau_1^\perp)-z|\geq1,$ 
\[
\int_{\tau_1^\perp}\int_{\tau_2^\perp}|s-t|^2d\HH^{n-1}(s)d\HH^{n-1}(t)\leq n\cdot |f_1(\tau_1 ,\tau_2)- f_2(\tau_1 ,\tau_2)|^2\leq C_3\cdot |f_1(\tau_1 ,\tau_2)- f_2(\tau_1 ,\tau_2)|^2.
\] This proves the lemma.
\end{proof}
By the definition of the multiplicity of $f,$ we have the following proposition.
\begin{pro}\label{prop:basic}
\[
\sum_{(\tau_1, \tau_2)\in (\partial\Omega)^2_g}|f_1(\tau_1 ,\tau_2)- f_2(\tau_1 ,\tau_2)|^2\leq m_f \cdot\sum_{z\in \delta\Omega} \sum_{\omega\in \delta\Omega}|z -\omega|^2.
\]
\end{pro}


\begin{lem} Let $\Omega$ be a finite subset of $\Z^n.$ Then
$$m_f\leq 8n^2.$$\label{L4ofn}
\end{lem}
\begin{proof}[Proof]
Fix $(x,y)\in f((\partial\Omega)^2_g).$ Let $(\tau_1,\tau_2)\in (\partial\Omega)^2_g$ such that $f(\tau_1,\tau_2)=(x,y).$ We estimate the number of possible pairs $(\tau_1,\tau_2)$ satisfying the above property. By the definition of $f$ in \eqref{def:f}, $f_1(\tau_1,\tau_2)=P(\tau_1^\perp)=x,$ which yields that $x$ is an end-vertex of the edge $\tau_1.$ Hence, the number of possible choices of $\tau_1$ is at most $2n.$ To determine the edge $\tau_2,$ it is divided into the following cases.
\begin{enumerate}
\item[\emph{Case~1.}] $P(\tau_1^\perp)\neq P(\tau_2^\perp).$ By \eqref{def:f},
$$f_2(\tau_1,\tau_2)=P(\tau_2^\perp)=y.$$ Hence, $y$ is an end-vertex of the edge $\tau_2.$ So that the number of possible choices of $\tau_2$ is at most $2n.$
\item[\emph{Case~2.}] $P(\tau_1^\perp)=P(\tau_2^\perp)\not\in \delta'\Omega.$ In this case, $P(\tau_2^\perp)=x$ and $x$ is an end-vertex of the edge $\tau_2.$ Hence, the number of possible choices of $\tau_2$ is at most $2n.$
\end{enumerate} Combining the above cases, the number of possible choices of $\tau_2$ is at most $4n.$ This yields that $$m_f\leq 2n\times 4n=8n^2.$$ This proves the lemma.
\end{proof}

The following lemma is a version of weighted isoperimetric inequality by symmetrization approaches, see e.g. \cite{betta1999}, 
which is crucial in the proof of Brock's result in $\R^n.$ For any $R>0,$ we denote by $B_R$ the ball of radius $R$ centered at the origin in $\R^n.$
\begin{lem}\label{lem:rearr} Let $U$ be a bounded domain in $\R^n$ with Lipschitz boundary. Let $R>0$ such that $\HH^n(B_R)=\HH^n(U).$ Let $g\in C([0,+\infty))$ be nonnegative, nondecreasing and suppose that 
\begin{equation*}
(g(z^{\frac{1}{n}})-g(0))z^{1-\frac1n},~z\geq 0 \label{convex}
\end{equation*}
is convex. Then
\begin{equation*}
\int_{\partial U}g(|x|)dS\geq\int_{\partial 
B_R}g(|x|)dS=n\omega _ng(R)R^{n-1}.
\end{equation*}
\end{lem}

Now we are ready to prove our main result. 
\begin{proof}[Proof of Theorem~\ref{mainthm:1}]


For $1\leq k\leq n,$ we denote by $E_k$ the edges in $E(\Omega, \bar{\Omega})$ parallel to $e_k,$ and by $|E_k|$ the cardinality of $E_k.$
Using the coordinate functions of $\R^n$, we define functions $u_i\in \R^{\Z^n}$ for $2\leq i\leq n+1$ as follows,
\[
u_i(z):=|E_{i-1}|^{-\frac{1}{2}}(z_{i-1}-\overline{z_{i-1}}),
\] where $$\overline{z_{i-1}}:=\frac{1}{| \delta\Omega|}\sum_{z\in \delta\Omega}z_{i-1}.$$
Then, we get
$$D_{\Omega}(u_i,u_j)=\delta_{ij},\ \sum_{z\in \delta\Omega}u_i(z)=0,\quad 2\leq i,j\leq n+1.$$ 
Hence, by the variational principle,  Theorem~\ref{thm:variation} with $p = n+1$,
\begin{eqnarray*}\sum^{n+1}_{i=2}\frac{1}{\lambda _i(\Omega)}
&\geq&\sum^{n+1}_{i=2}\sum_{z\in \delta\Omega}u^2_i(z)
=\sum^{n+1}_{i=2}\sum_{z\in \delta\Omega}\frac{|z_{i-1}-\overline{z_{i-1}}|^2}{|E_{i-1}|}\\
&\geq&  \frac{1}{|E(\Omega, \bar{\Omega})|} \sum^{n+1}_{i=2}\sum_{z\in \delta\Omega}|z_{i-1}-\overline{z_{i-1}}|^2.\\
\end{eqnarray*}
Note that
\begin{eqnarray*}
&&\sum_{z\in \delta\Omega} \sum_{\omega\in \delta\Omega}|z -\omega|^2 \\
&=&\sum_{z\in \delta\Omega} \sum_{\omega\in \delta\Omega}\sum^{n+1}_{i=2}|z_{i-1}-\omega_{i-1}|^2 = \sum^{n+1}_{i=2} \left(2|\delta\Omega|\sum_{z\in \delta\Omega} z^2_{i-1}-2(\sum_{z\in \delta\Omega}z_{i-1})^2\right)\\
&= &2|\delta\Omega|\sum^{n+1}_{i=2}\left(\sum_{z\in \delta\Omega}z^2_{i-1}-\frac{(\sum_{z\in \delta\Omega}z_{i-1})^2}{| \delta\Omega|} \right)= 2|\delta\Omega|\sum_{z\in \delta\Omega}\sum^{n+1}_{i=2}(z^2_{i-1}-2\overline{z_{i-1}}\cdot z_{i-1}+\overline{z_{i-1}}^2) \\
&= &2|\delta\Omega|\sum_{z\in \delta\Omega}\sum^{n+1}_{i=2}(z_{i-1}-\overline{z_{i-1}})^2.
\end{eqnarray*}
Hence, we get
\begin{equation}
\sum^{n+1}_{i=2}\frac{1}{\lambda _i(\Omega)}\geq \frac{1}{2|\delta\Omega|}\cdot  \frac{1}{|E(\Omega, \bar{\Omega})|} \cdot \sum_{z\in \delta\Omega} \sum_{\omega\in \delta\Omega}|z -\omega|^2. \label{main}
\end{equation}


In the following, we compare the right hand side of the above inequality, a discrete summation, with a continuous quantity,
$$\int_{\partial(\hat{\Omega})}\int_{\partial(\hat{\Omega})}|s-t|^2d\HH^{n-1}(s)d\HH^{n-1}(t).$$
By \eqref{eq:pho} and the decomposition of pairs of boundary edges into ``good" and ``bad" ones, \eqref{eq:gb}, we have
\begin{align*}
\int_{\partial(\hat{\Omega})}\int_{\partial(\hat{\Omega})}|s-t|^2d\HH^{n-1}(s)d\HH^{n-1}(t) 
&= \sum_{(\tau_1, \tau_2)\in \partial\Omega\times\partial \Omega}\int_{\tau_1^\perp}\int_{\tau_2^\perp}|s-t|^2d\HH^{n-1}(s)d\HH^{n-1}(t)\\
&= \left(\sum_{(\tau_1, \tau_2)\in (\partial\Omega)^2_g}+ \sum_{(\tau_1, \tau_2)\in (\partial\Omega)^2_b}\right)\int_{\tau_1^\perp}\int_{\tau_2^\perp}|s-t|^2d\HH^{n-1}(s)d\HH^{n-1}(t)\\ 
&=: \RNum{1}+\RNum{2}.
\end{align*}

By Lemma~\ref{fC1} and Proposition~\ref{prop:basic}, 
\begin{eqnarray*}
\RNum{1}&=&\sum_{(\tau_1, \tau_2)\in (\partial\Omega)^2_g}\int_{\tau_1^\perp}\int_{\tau_2^\perp}|s-t|^2d\HH^{n-1}(s)d\HH^{n-1}(t)\leq C_3\cdot m_f \sum_{z\in \delta\Omega} \sum_{\omega\in \delta\Omega}|z -\omega|^2.\end{eqnarray*}
For any $(\tau_1, \tau_2)\in (\partial\Omega)^2_b,$ the inequality (\ref{3.3}) yields 
\[
\int_{\tau_1^\perp}\int_{\tau_2^\perp}|s-t|^2d\HH^{n-1}(s)d\HH^{n-1}(t) \leq n.
\]
Combining this with Proposition~\ref{prop:bpv1}, we get
\[
\RNum{2}\leq 4n^3|\delta'\Omega|.
\]
Hence, we obtain
\begin{equation}
\int_{\partial(\hat{\Omega})}\int_{\partial(\hat{\Omega})}|s-t|^2d\HH^{n-1}(s)d\HH^{n-1}(t)=\RNum{1}+\RNum{2}
\leq C_3\cdot m_f \sum_{z\in \delta\Omega} \sum_{\omega\in \delta\Omega}|z -\omega|^2 +4n^3|\delta'\Omega|.
\label{right}
\end{equation}
By the calculation,
\begin{equation}\label{eq:dev1}
\int_{\partial(\hat{\Omega})}\int_{\partial(\hat{\Omega})}|s-t|^2d\HH^{n-1}(s)d\HH^{n-1}(t)=2\HH^{n-1}(\partial(\hat{\Omega}))\int_{\partial(\hat{\Omega})}|s-c|^2d\HH^{n-1}(s),
\end{equation}
where $c=\frac{1}{\HH^{n-1}(\partial(\hat{\Omega}))}\int_{\partial(\hat{\Omega})}sd\HH^{n-1}(s)$ is the barycenter of $\partial(\hat{\Omega}).$ 
We denote by $\Omega_0:=\hat{\Omega}-c$ the translation of the domain $\hat{\Omega}$ by a vector $-c.$ It is easy to see that the barycenter of the boundary of $\Omega_0$ is the origin, so that
\[
\int_{\partial(\hat{\Omega})}|s-c|^2d\HH^{n-1}(s)=\int_{\partial\Omega_0}|s|^2d\HH^{n-1}(s).
\]

Since $\HH^n(\Omega_0)=\HH^n(\hat{\Omega})=|\Omega|,$ we choose $R>0,$ such that
$$\HH^n(B_R)=\omega_n R^n=|\Omega|=\HH^n(\Omega_0).$$
By Lemma \ref{lem:rearr}, taking $U=\Omega_0$ and $g(t)=t^2,$ we get
\[
\int_{\partial\Omega_0}|s|^2d\HH^{n-1}(s)\geq \int_{\partial B_R}|s|^2d\HH^{n-1}(s)=n\omega_nR^{n+1}=n|\Omega| R.
\]
By the above estimates, noting that $\HH^{n-1}(\partial(\hat{\Omega}))=|\partial \Omega|,$ we obtain
\begin{equation}
\int_{\partial(\hat{\Omega})}\int_{\partial(\hat{\Omega})}|s-t|^2d\HH^{n-1}(s)d\HH^{n-1}(t)\geq2n|\partial\Omega|\cdot|\Omega|R.
\label{left}
\end{equation}
By the equations (\ref{right}) and (\ref{left}), 
\begin{equation}\label{eq:est1}
\sum_{z\in \delta\Omega} \sum_{\omega\in \delta\Omega}|z -\omega|^2 \geq \frac{2n(|\partial\Omega|\cdot |\Omega|R-2n^2|\delta'\Omega|)}{C_3\cdot m_f}.
\end{equation}
Note that
\[
|\delta'\Omega|\leq |\delta\Omega|\leq |\partial\Omega|\leq 2n|\Omega|,
\]
and
\[
2|E(\Omega, \bar{\Omega})|=2n\cdot |\Omega|+|\partial\Omega|.
\] This yields that  $$|E(\Omega, \bar{\Omega})|=n|\Omega|+\frac{1}{2}|\partial\Omega|\leq 2n|\Omega|.$$
By (\ref{main}) and \eqref{eq:est1},
\begin{eqnarray*}
\sum^{n+1}_{i=2}\frac{1}{\lambda _i(\Omega)}
&\geq& \frac{2n(|\partial\Omega|\cdot |\Omega|R-2n^2|\delta'\Omega|)}{2|\delta\Omega|\cdot|E(\Omega, \bar{\Omega})|C_3\cdot m_f}.
\end{eqnarray*}

This yields that
\begin{eqnarray*}
\sum^{n+1}_{i=2}\frac{1}{\lambda _i(\Omega)}
&\geq& \frac{n(|\partial\Omega|\cdot |\Omega|R-2n^2|\delta'\Omega|)}{|\partial\Omega|\cdot 2n|\Omega|C_3\cdot m_f}\\
&\geq&\frac{1}{2C_3\cdot m_f}\left(R-\frac{2n^2}{|\Omega|}\right)=\frac{1}{2C_3\cdot m_f}\left((\omega_n^{-1}|\Omega|)^{\frac1n}-\frac{2n^2}{|\Omega|}\right).
\end{eqnarray*}
By Lemma~\ref{L4ofn}, we get
$$\sum^{n+1}_{i=2}\frac{1}{\lambda _i(\Omega)}
\geq C_1|\Omega|^{\frac1n}-\frac{C_2}{|\Omega|},$$
where $C_1=(64 n^3 \omega_n^{\frac1n})^{-1}, C_2=\frac{1}{32n}.$
This proves the theorem.
\end{proof}

We construct an infinite graph with bounded degree, which doesn't satisfy the property \eqref{eq:to0}. For any $n\in \N,$ we denote by $T_n$ the graph obtained by a complete binary tree of depth $n$ with a pending vertex $P_{T_n}$ attaching to the root of the tree, see Figure~\ref{fig:K_n}.  Let $T_n'$ be a copy of $T_n.$ Let $K_n$ be the graph constructed by the disjoint union of $T_n$ and $T_n'$ by identifying the leaves of $T_n$ with those of $T_n'$ accordingly. Note that $K_n$ has two pending vertices $P_{T_n}$ and $P_{T_n'},$ see Figure~\ref{fig:K_n}.
\begin{figure}[htbp]
 \begin{center}
   \begin{tikzpicture}
    \node at (0,0){\includegraphics[width=0.5\textwidth]{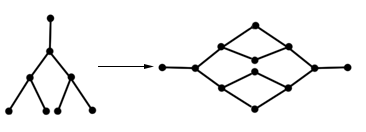}};
    \node at (-2.8, -1.5){\Large $T_2$};
    \node at (1.7, -1.5){\Large $K_2$};
        \node at (-0.2, 0.5){\Large $P_{T_2}$};
    \node at (3.6, 0.5){\Large $P_{T_2'}$};
   \end{tikzpicture}
  \caption{$K_2$ is constructed by $T_2$ and its copy $T_2'$. }
\label{fig:K_n}
 \end{center}
\end{figure}
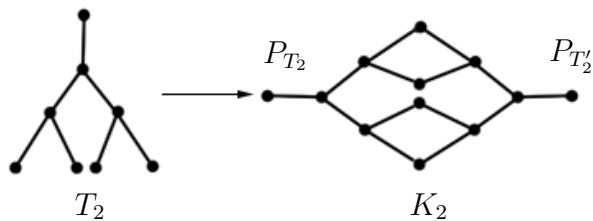
\begin{figure}[htbp]
 \begin{center}
   \begin{tikzpicture}
    \node at (0,0){\includegraphics[width=0.9\textwidth]{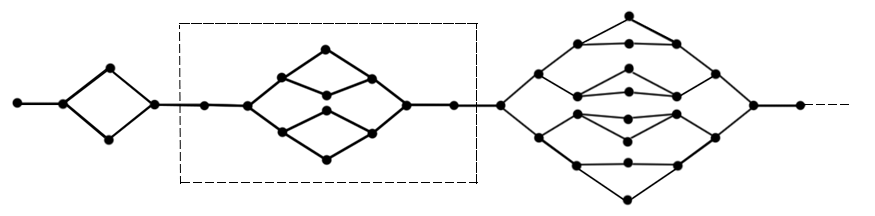}};
    \node at (-2.0, -2.0){\Large $K_2$};
     \node at (3.1, -2.0){\Large $K_3$};
   \end{tikzpicture}
  \caption{An infinite graph doesn't satisfy the property \eqref{eq:to0}.}
\label{fig:tree}
 \end{center}
\end{figure}
\begin{exa}\label{tree} Let $G=(K_1\cup K_2\cup K_3\cup\cdots)/\sim $ be an infinite graph obtained by a disjoint union of $\{K_n\}_{n=1}^\infty$ with sequentially identifying pending vertices, $P_{T_n'}$ and $P_{T_{n+1}}$, of $K_n$ and $K_{n+1},$ $n\geq 1.$ Let $\Omega_i=K_i,$ $i\geq 1,$ where $K_i$ is the embedded image of $K_i$ in $G.$ Then $$\lambda_2(\Omega_i)=\frac{2}{\mathrm{Res}(\overline{K_i})}=\frac{2}{6-\frac{1}{2^{i-1}}},\quad\forall i\geq 2,$$ where $\mathrm{Res}(\overline{K_i})$ is the effective resistance between the vertices in $\delta K_i$ in the induced subgraph on $\overline{K_i},$ for which each edge is interpreted as a resistor of resistance one; see \cite[Definition~2.7, p.42]{Barlow2017book} for the definition. This yields that
$$\lambda_2(\Omega_i)\geq c>0,\quad i\geq 1.$$
\end{exa}


\bigskip

\textbf{Acknowledgements.} 
We thank the anonymous referees for providing helpful comments and suggestions to improve the writing of the paper, in particular pointing out some useful references on Steklov eigenvalues. We thank Zuoqin Wang for many helpful suggestions on the Steklov problem on graphs. 

B. H. is supported by NSFC (China), grant no. 11831004 and no. 11826031.


\bibliography{subgraphslattices}
\bibliographystyle{alpha}

\end{document}